\documentclass[12pt,reqno]{amsart}
\usepackage{amssymb}
\input amssym.def
\usepackage{amsmath,amsfonts,hyperref,xcolor,textcomp}
\usepackage{amscd,orcidlink}
\usepackage[mathscr]{eucal}
\usepackage{graphicx}
\usepackage{tikz}
\usepackage{float}
\hypersetup{colorlinks=true,citecolor={purple},linkcolor={blue},urlcolor={violet}}

\makeatletter
\@namedef{subjclassname@2020}{%
  \textup{2020} Mathematics Subject Classification}
\makeatother

\newcommand{\F}{{\mathbb F}}
\newcommand{\Z}{{\mathbb Z}}

\newcommand{\C}{{\mathbb C}}
\newcommand{\uH}{{\mathbb H}}

\newtheorem{thm}{Theorem}
\newtheorem{lem}{Lemma}
\newtheorem{cor}{Corollary}

\newtheorem{rmk}{Remark}

\newcommand{\thmref}[1]{Theorem~\ref{#1}}

\newcommand{\lemref}[1]{Lemma~\ref{#1}}
\newcommand{\corref}[1]{Corollary~\ref{#1}}
\newcommand{\rmkref}[1]{Remark~\ref{#1}}

\begin{document}

\title[On the monomials of Poincar\'e Series of negative index]{On the monomials of Poincar\'e series of negative index}

\author{Divyanshu Kala \orcidlink{0009-0007-9463-9997}}

\address{Divyanshu Kala\\ \newline
Institute of Mathematical Sciences, IV Cross Street, Taramani, Chennai 600113, India.}
\email{divyanshuk@imsc.res.in}



\date{\today}

\begin{abstract}
It is well known that an Eisenstein series $E_k$ cannot be written as a product of two lower-weight Eisenstein series, except 
for $E_{14}=E_4 E_{10}= E_4^2 E_6= E_6 E_8$. A similar result is also known for the Hecke eigenforms. It is also known that equalities among the monomials of the Eisenstein series can be reduced to the above-mentioned identities. Recently, the present author, along with E. Saha studied the possible monomial relations of the Poincar\'e cusp forms of index $1$. So far, the problem of studying the monomial relations has been restricted to the setup of holomorphic modular forms. For an even integer $k\ge 4$ and a negative integer $m$, the Poincar\'e series $G_{k}(z,m)$ of weight $k$ and index $m$ is a weakly holomorphic modular form having a pole of order $-m$ at $\iota\infty$. It is immediate that the Poincaré series $G_k(z,m)$ for $m<0$ can
not be written as a product of two lower-weight Poincar\'e series of the same index $m$. In this article, we investigate the possible equalities among the monomials of the Poincaré series 
 $G_k(z,m)$ for arbitrary $m<0$. In particular, we show that for any $m<0$ such that $0.06\le\{-2m\pi\}\le0.99$, two 
 monomials composed of $G_k(z,m)$ of the weights $k\ge 50$ are never equal, where $\{x\}$ denotes the fractional part of a real number $x$. In view of Weyl's equidistribution criterion, at least 
 $93\%$ of $m<0$ satisfies the above inequality. 
 \end{abstract}

\maketitle

\section{Introduction}

Studying the possible relations among the monomials of modular forms is a problem of classical interest, and it has also attracted considerable attention in recent times. Let $\uH:=\{z\in\C: \Im(z)>0\}$. For $z\in \uH$, the (normalised) Eisenstein series for ${\rm SL}(2,\Z)$ of weight $k\ge 4$ is defined by the absolutely convergent series
$$
E_k(z):=\frac{1}{2}\sum_{c,d\in\Z \atop \gcd(c,d)=1} \frac{1}{(cz+d)^k}.
$$
The classical identity among the divisor functions
$$
\sigma_7(n)=\sigma_3(n)+120\sum_{r=1}^{n-1}\sigma_3(r)\sigma_3(n-r),
$$
where $\sigma_{k-1}(n)=\sum_{d\mid n}d^{k-1}$, is equivalent to $E_8=E_4^2$, which is a consequence of the 
fact that the $\C$-vector space of modular forms of weight $8$ is $1$-dimensional. In fact each of the $\C$-vector spaces of modular forms of weight $4,6,8,10$ and $14$ is also $1$-dimensional, which gives rise to the equalities $E_{14}=E_4 E_{10}= E_4^2 E_6= E_6 E_8$. This naturally leads to the question of whether an Eisenstein series can be written as a product of some lower-weight Eisenstein series. Using the Rankin-Selberg method, W. Duke \cite{WD} and E. Ghate \cite{EG} independently proved that an Eisenstein series is never a product of two lower-weight Eisenstein series, except for the above-mentioned identities. They also proved that a Hecke eigenform is not a product of two lower-weight Hecke eigenforms, apart 
from $16$ exceptions. Later, E. Ghate \cite{EG2} extended this result to the Hecke eigenforms of square-free level, and M. L. Johnson \cite{J} extended it further 
to arbitrary level. For level-one eigenforms, one can note that the product of at least two Hecke eigen cusp forms can never be a Hecke eigenform, as the order of 
zero of a Hecke eigenform at infinity is at most $1$. In this direction, B. A. Emmons and D. Lanphier \cite{EL} proved that products of an arbitrary number of level one Hecke eigenforms can never be a Hecke eigenform.  D. Lanphier and R. Takloo-Bighash \cite{LT} studied this problem for the Rankin-Cohen brackets of eigenforms. Recently, the finiteness of the Hecke eigenform identities has also been studied for Hilbert modular forms. K. Joshi and Y. Zhang \cite{JZ} proved the finiteness of Hilbert eigenform identities over real quadratic fields. Y. You and Y. Zhang \cite{YZ} extended this result over the totally real fields.

A more general question in this direction is to study the relations among the monomials of Eisenstein series. Using a completely different approach via the location and separation properties of the zeros of the Eisenstein series in the standard fundamental domain, T. Griffin, N. Kenshur, A. Price, B. Vandenberg-Daves, H. Xue and D. Zhu \cite{GKP} proved the following on the monomial relations of the Eisenstein series.
For even integers $k_1,\ldots,k_{n_1},l_1,\ldots,l_{n_2}\ge4$, the equation
 $$
 \prod_{i=1}^{n_1} E_{k_i}=\prod_{j=1}^{n_2} E_{l_j},
 $$
 has no solutions other than those arising from the dimension restrictions.
They showed that if $k=\max\{k_i,l_j: 1\le i\le n_1, 1\le j \le n_2\}$ is more than $14$, then there is a zero of $E_k$ which is not a zero of $E_t$ for any $t<k$, and hence such a monomial relation cannot exist. The location of the zeros of $E_k$ in the standard fundamental domain
$$
\F:= \left\{ |z| \ge 1, -\frac{1}{2} \le \Re(z) \le 0 \right\}
\cup \left\{ |z| > 1, 0 < \Re(z) < \frac{1}{2} \right\}
$$ 
is known to be on the arc $A:=\{e^{\iota\theta}:\theta\in[\pi/2,2\pi/3]\}$ due to F. K. C. Rankin and H. P. F. Swinnerton-Dyer \cite{RSD}.

Poincar\'e series can be viewed as a generalisation of the Eisenstein series. For an even integer $k\ge 4$, $m\in \Z$ and $z\in \uH$, it is defined by
$$
G_k(z,m):=\frac{1}{2}\sum_{c,d\in\Z \atop \gcd(c,d)=1} \frac{e^{2\pi \iota m\gamma(z)}}{(cz+d)^k},
$$
where $\gamma\in{\rm SL}(2,\Z)$ with the second-row as $(c\ \ d)$. For $m=0$, it is $E_k$ and for $m\ge1$, it is a cusp 
form. For $m<0$, it is a weakly holomorphic modular form having a pole of order $-m$ at $\iota\infty$. 
For the Poincaré series, only partial information about the location of its zeros is known. For $m\le 1$, R. A. Rankin 
\cite{RR} proved that all the zeros of $G_k(z,m)$ in $\F$ are on the arc $A$. Using the location and separation
 properties of the zeros of $G_{k}(z,1)$ on the arc $A$, the author of this article and E. Saha \cite{KS2} proved that for $1\le i\le n_1, 1\le j\le n_2$ and $k_i\ne l_j \ge 36$,
 $$
\prod_{i=1}^{n_1}G_{k_i}(z,1)\neq\prod_{j=1}^{n_2}G_{l_j}(z,1).
 $$
 
 So far, the problem of studying the monomial relations has been limited to the 
 setup of holomorphic modular forms. Unlike the space of modular forms 
 being finite-dimensional, the space of weakly holomorphic modular forms is 
 an infinite-dimensional $\C$-vector space. Poincar\'e series of negative index are important objects in the space of weakly holomorphic modular forms and they form a basis of this vector space. Also, for $m<0$, the location of the zeros of $G_k(z,m)$ in $\F$ is known to be on the arc $A$. 
 Therefore, a natural curiosity here is to investigate the possible monomial relations among $G_k(z,m)$ for $m<0$.
However, opposed to the previous cases $m=0$ and 
$m=1$, here we have to deal with arbitrary $m<0$, leading us to certain complications and restrictions. For a real number $x$, let $\{x\}$ denote the fractional part of $x$. In this article, we prove the following theorem.
\begin{thm}\label{main}
Let $s,t,m\in\Z$ be such that $s,t\ge1$, $m<0$ and $0.06\le\{-2m\pi\}\le0.99$. 
For any even integers $k_i,l_j\ge 50$ such that $k_i\neq l_j$ for every $1\le i\le s, 1\le j\le t$,
we have,
$$
\prod_{i=1}^sG_{k_i}(z,m)\neq\prod_{j=1}^tG_{l_j}(z,m).
$$
\end{thm}

Before we go into more details of our theorem, we recall that a sequence of real numbers $(x_n)_{n\ge1}$ is said to be uniformly distributed modulo $1$ if for every pair of real numbers $a,b$
with $0\le a<b\le 1$, we have 
$$
\lim_{N\to\infty}\frac{\#\{n\le N: \{x_n\}\in[a,b]\}}{N}=b-a.
$$
Weyl's equidistribution criterion \cite[Theorem 11.1.5]{RM} tells us that a sequence $(x_n)_{n\ge1}$ is uniformly 
distributed modulo $1$ if and only if 
$$
\sum_{n=1}^{N}e^{2\pi\iota r x_n}=o(N), \ \ \ r=\pm1,\pm2,\pm3,\ldots.
$$
In view of Weyl's equidistribution criterion, for any irrational number $\theta$, the sequence 
$(n\theta)_{n\ge1}$ is uniformly distributed modulo 1; in particular, the sequence 
$(\{-2m\pi\})_{m\le-1}$ is uniformly distributed in the interval $[0,1)$. Therefore at least $93\%$ of $m\le 0$
 satisfies the condition $0.06\le \{-2m\pi\}\le 0.99$.
 
Our proof of \thmref{main} is constituted from the ideas of the proofs of \cite[Theorem 1.1]{GKP} and of
\cite[Theorem 1]{KS2}, where this question was studied for the cases $m=0$ and $m=1$, respectively. However, unlike the previous 
cases, here we are dealing with varying $m<0$, which poses significant challenges.
 To shed some light on the difficulties in dealing with varying $m<0$, we give the structure of the proof of \thmref{main}, comparing with the proofs of \cite[Theorem 1.1]{GKP} and
\cite[Theorem 1]{KS2}, in the next section.

\section{Notations and structure of the proof}\label{struc}
Let $m<0$ be an integer. Let $I:=[\pi/2,2\pi/3]$, $I^{\circ}:=(\pi/2,2\pi/3)$ and $m(k)$ be the integer defined by writing the even integer $k$ as $k=12m(k)+s$, for $s\in\{0,4,6,8,10,14\}$. R. A. Rankin \cite{RR} proved that all the zeros of $G_k(z,m)$ in $\F$ are on the arc $A$. Further, exactly $m(k)-m$ many zeros of $G_k(z,m)$ in $\F$ are on the arc $A^{\circ}:=\{e^{\iota\theta}:\theta\in(\pi/2,2\pi/3)\}$.
He proved it by showing that for $\theta\in I^{\circ}$, the real valued function $H_{k,m}(\theta):=e^{2\pi m\sin\theta}e^{\iota k\theta/2}G_k(e^{\iota\theta},m)$ 
changes sign exactly $m(k)-m+1$ many times in the interval $I^{\circ}$, giving the existence of at least $m(k)-m$ many zeros on $I^{\circ}$, which is exact 
due to the valence formula. Following \cite[\S2]{KS1}, we write $H_{k,m}(\theta)$ as
\begin{equation}\label{hkm}
H_{k,m}(\theta)=2\cos b_{k,m} (\theta)+T_{k,m}(\theta),
\end{equation}
where
$$
T_{k,m}(\theta)=
(-1)^m P_{k,m}(\theta)
+(-1)^mQ_{k,m}(\theta)+e^{2\pi m \sin\theta}R_{k,m}(\theta),
$$
$$
b_{k,m}(\theta)=\frac{k\theta}{2} +2m\pi\cos\theta,\ \
P_{k,m}(\theta)=\frac{e^{m\pi(2\sin\theta-\tan(\theta/2))}}{(2\cos(\theta/2))^k}, \ \  Q_{k,m}(\theta)= \frac{e^{m\pi(2\sin\theta-\cot(\theta/2))}}{(2i\sin(\theta/2))^k},
$$
and 
$$
R_{k,m}(\theta):=\frac{1}{2}\sum_{c,d\in\Z \atop {c^2+d^2 \ge 5 \atop \gcd(c,d)=1}} 
\frac{e^{2\pi i m \gamma(e^{i\theta})}}{(ce^{i\frac{\theta}{2}}+de^{-i\frac{\theta}{2}})^k}.
$$
We note the following facts about these functions from \cite[\S2]{KS1}.
\begin{itemize}
\item[(a)] The function $b_{k,m}(\theta)$ is an increasing function on $I$.
\item[(b)] The function $P_{k,m}(\theta)$ is an increasing function on $I$.
\item[(c)] On $I$, the function $Q_{k,m}(\theta)$ satisfies the inequality
\begin{equation}\label{qkm}
|Q_{k,m}(\theta)| \le\frac{e^{m\pi}}{2^{k/2}}.
\end{equation}
\item[(d)] On $I$, we have
\begin{equation}\label{rkm}
|e^{2 \pi m\sin\theta}R_{k,m}(\theta)| \le \frac{0.359e^{1.039m\pi}}{2^{k/2}}.
\end{equation}
\end{itemize}

R. A. Rankin \cite{RR} established that the first $m(k)-m$ many zeros of $\cos b_{k,m}(\theta)$ in $I^{\circ}$ corresponds to the zeros of 
$H_{k,m}(\theta)$ in $I^{\circ}$. Let 
$$
\alpha_{k,m}^{(1)}<\alpha_{k,m}^{(2)}<\cdots<\alpha_{k,m}^{(m(k)-m)}
$$  
denote the first $m(k)-m$ many zeros of $\cos b_{k,m}(\theta)$ in $I^{\circ}$ and 
\begin{equation}\label{order}
\tilde{\alpha}_{k,m}^{(1)}<\tilde{\alpha}_{k,m}^{(2)}<\cdots<\tilde{\alpha}_{k,m}^{(m(k)-m)}
\end{equation} 
denote the corresponding zeros of  $H_{k,m}(\theta)$ in $I^{\circ}$. Let $r_k$ denote the remainder of the integer $k$ modulo $4$. Since $b_{k,m}(\pi/2)=k\pi/4$, for $1\le i\le m(k)-m$, we note the following
\begin{equation}\label{bkma}
b_{k,m}(\alpha_{k,m}^{(i)})=\frac{k\pi}{4}+\frac{r_k\pi}{4}+\frac{\pi}{2}+(i-1)\pi.
\end{equation}

\subsection{Structutre of the proof} For any even integer, $l\ge50$ and $m\le -1$ satisfying 
$0.06\le\{-2m\pi\}\le0.99$, the key idea of the proof is to find a zero of $G_l(z,m)$ which is not a zero of $G_k(z,m)$ for any even integer
$50\le k<l$. In view of this, we assume that for even integers $k_i,l_j\ge 50$ such that $k_i\neq l_j$ for every $1\le i\le s, 1\le j\le t$ and $m\le -1$ satisfying 
$0.06\le\{-2m\pi\}\le0.99$, we have 
$$
\prod_{i=1}^sG_{k_i}(z,m)=\prod_{j=1}^tG_{l_j}(z,m).
$$
Without loss of generality, suppose that  $l_1=\max\{k_i,l_j:1\le i\le s,1\le j\le t\}$. Then for $1\le i\le s$, we have $l_1>k_i$. Therefore
there is a zero of $G_{l_1}(z,m)$ which is not a zero of any of $G_{k_i}(z,m)$ for any $1\le i\le s$, which is a contradiction and this proves
\thmref{main}. The similar idea was used for the proof of  \cite[Theorem 1.1]{GKP} and \cite[Theorem 1]{KS2} for $m=0$ and $m=1$, respectively. 
Below, we give the 
key steps used in the proof of \cite[Theorem 1]{KS2} to find a zero of $G_l(z,1)$ different from all the 
zeros of $G_k(z,1)$ for any even integers $36\le k<l$.

For any even integers $l>k\ge36$ such that $(r_k,r_l)=(2,0)$ and $r_k=r_l$, it was establised in \cite{KS2} that 
\begin{equation}\label{m=1}
\tilde{\alpha}_{l,1}^{(1)}<\tilde{\alpha}_{k,1}^{(1)}<\tilde{\alpha}_{k,1}^{(2)}<\cdots<\tilde{\alpha}_{k, 1}^{(m(k)-1)}.
\end{equation}
However the case $(r_k,r_l)=(0,2)$ was divided into two parts $k< l/2+6$ and $ l/2+6<k<l$. For $k< l/2+6$ the ordering \eqref{m=1} holds, and for $l/2+6<k<l$, one has
$$
\tilde{\alpha}_{k,1}^{(1)}<\tilde{\alpha}_{l,1}^{(1)}<\tilde{\alpha}_{k,1}^{(2)}<\cdots<\tilde{\alpha}_{k, 1}^{(m(k)-1)}.
$$
This proves the existence of a zero of $G_l(z,1)$ distinct from all the zeros of $G_k(z,1)$ for $l>k\ge 36$.

For the proof of \thmref{main}, in the cases $(r_k,r_l)=(2,0)$ and $r_k=r_l$, we prove
$\tilde\alpha_{l,m}^{(1)}<\tilde\alpha_{k,m}^{(1)}$, and this is done in \S\ref{pmain}. However, the case $(r_k,r_l)=(0,2)$ is not fairly straightforward, as in this for some even $k<l$, there is no natural ordering among
$\tilde\alpha_{k,m}^{(1)}$ and $\tilde\alpha_{l,m}^{(1)}$. If there exists a $\lambda_{l,m}\in \Z$ such that $l/2-\lambda_{l,m}<l$, $r_{(l-2\lambda_{l,m})/2}=0$ and
\begin{align*}
    &\tilde \alpha_{(l-2\lambda_{l,m})/2,m}^{(1)}<\tilde\alpha_{l,m}^{(1)},\ \ \text{for} \ \ l/2-\lambda_{l,m}\ge 50  \ \ \\
&\tilde \alpha_{(l-2\lambda_{l,m}-8)/2,m}^{(1)}<\tilde\alpha_{l,m}^{(1)},\ \ \text{for} \ \  l/2-\lambda_{l,m}-4\ge 50,  \ \ 
\end{align*}
then for any even $50\le k<l$ such that $(r_k,r_l)=(0,2)$, we have
\begin{align*}
&\tilde\alpha_{l,m}^{(1)}<\tilde\alpha_{(l-2\lambda_{l,m}-8)/2,m}^{(1)}\le\tilde\alpha_{k,m}^{(1)} \ \ \text{for}\ \
k\le (l-2\lambda_{l,m}-8)/2, \\
&\tilde\alpha_{k,m}^{(1)}\le \tilde\alpha_{(l-2\lambda_{l,m})/2,m}^{(1)}<\tilde\alpha_{l,m}^{(1)}
\ \ \text{for}\ \ k\ge (l-2\lambda_{l,m})/2.
\end{align*}
Since $\tilde\alpha_{l,m}^{(1)}$
changes with $l$ and $m$, the integer $\lambda_{l,m}$ depends on both $l$ and $m$.
In \cite{KS2}, for the case of $m=1$, the choice of $\lambda_{l,m}$ was $-7$ or $-9$, and in \cite{GKP}, for the case of $m=0$, it was $-1$ or $-3$. Since in this case we are dealing with an arbitrary $m<0$, there is no obvious 
choice for $\lambda_{l,m}$ and it needs to be computed separately. One of the key 
obstacles in doing so is that there is no closed formula for the zeros of 
$\cos b_{k,m}(\theta)$ in $I^{\circ}$, which makes the study more complicated here.
 The integer $\lambda_{l,m}$ for any $l\ge 50$ and $m\le-1$, satisfying $0.06\le \{-2m\pi\}\le0.99$, is introduced in the next section. With this value of $\lambda_{l,m}$ we prove the 
following ordering in \S\ref{pmain}. For $l>k\ge50$ be even integers
\begin{itemize}
\item For $(r_k,r_l)=(2,0)$ and $r_k=r_l$, we prove
$$
\tilde\alpha_{l,m}^{(1)}<\tilde\alpha_{k,m}^{(1)}<\tilde\alpha_{k,m}^{(2)}<\cdots<\tilde\alpha_{k,m}^{(m(k)-m)}.
    $$
    \item  For $(r_k,r_l)=(0,2)$ with $50\le k\le l/2-\lambda_{l,m}-4<l$, we prove
    $$
\tilde\alpha_{l,m}^{(1)}<\tilde\alpha_{k,m}^{(1)}<\tilde\alpha_{k,m}^{(2)}<\cdots<\tilde\alpha_{k,m}^{(m(k)-m)},
    $$
and for $50\le l/2-\lambda_{l,m}\le k<l$, we prove
$$
\tilde\alpha_{k,m}^{(1)}<\tilde\alpha_{l,m}^{(1)}<\tilde\alpha_{k,m}^{(2)}<\cdots<\tilde\alpha_{k,m}^{(m(k)-m)}.
$$
\end{itemize}
This gives us a zero of $G_l(z,m)$, which is not a zero of $G_k(z,m)$ for any $50\le k<l$ and $m\le-1$ satisfying $0.06\le \{-2m\pi\}\le 0.99$.

\section{Computation of $\lambda_{l,m}$}
Since the zeros of $H_{k,m}(\theta)$ correspond to the zeros of $\cos b_{k,m}(\theta)$ in $I^{\circ}$, to show the 
required ordering among the zeros of $H_{k,m}(\theta)$ and $H_{l,m}(\theta)$, first we prove it
among the zeros of $\cos b_{k,m}(\theta)$ and $\cos b_{l,m}(\theta)$. Let $m<0$ be such that $0.06\le\{-2m\pi\}\le0.99$ and $l\ge 50$ be even such that 
$r_l=2$. The goal of this section is to find an integer $\lambda_{l,m}$  such that $l/2-\lambda_{l,m}<l$, $r_{(l-2\lambda_{l,m})/2}=0$ and
$$
    \alpha_{(l-2\lambda_{l,m})/2,m}^{(1)}<\alpha_{l,m}^{(1)}, \ \ \text{for}\ \ l/2-\lambda_{l,m}\ge 50,\ \  \text{and}
    $$
    $$
\alpha_{(l-2\lambda_{l,m}-8)/2,m}^{(1)}>\alpha_{l,m}^{(1)},    \ \ \text{for}\ \ l/2-\lambda_{l,m}-4\ge50.
$$

The primary difficulty here is that we do not have a closed formula for the 
zeros of $\cos b_{k,m}(\theta)$, unlike the case $m=0$.
Since we do not have an exact formula for the zeros of $\cos b_{k,m}(\theta)$ in $I^{\circ}$ first we obtain 
the bounds for $\alpha_{k,m}^{(1)}$. The linear approximation used in \cite{KS2} does not seem to give very good bounds here, especially the upper bound. Therefore, we do one more step after obtaining bounds via the linear approximation to get the final upper bound. 
First, we give the following general upper bound 
on $\alpha_{k,m}^{(i)}$, for $1\le i\le m(k)-m$, which is useful for the final computation of upper and lower bounds of $\alpha_{k,m}^{(1)}.$

\begin{lem}\label{1st zero}
Let $k'>k\ge50$ even integers be such that $r_k=r_{k'}$ and $m'\le m\le -1$. For $1\le i\le m(k)-m$ we have $\alpha_{k',m'}^{(i)}<\alpha_{k,m}^{(i)}$.
\end{lem}
\begin{proof}
We consider the following, 
$$
b_{k,m}(\alpha_{k,m}^{(i)})-b_{k,m}(\alpha_{k',m'}^{(i)})=
b_{k,m}(\alpha_{k,m}^{(i)})-b_{k',m'}(\alpha_{k',m'}^{(i)})+\frac{(k'-k)\alpha_{k',m'}^{(i)}}{2}+2\pi(m'-m)\cos\alpha_{k',m'}^{(i)}.
$$
Since $r_k=r_{k'}$, using \eqref{bkma}, we have
$$ 
b_{k,m}(\alpha_{k,m}^{(i)})-b_{k,m}(\alpha_{k',m'}^{(i)})=\frac{(k'-k)}{2}\left(\alpha_{k',m'}^{(i)}-\frac{\pi}{2}\right)+2\pi(m'-m)\cos\alpha_{k',m'}^{(i)}.
$$
Since $m'\le m$, $k'>k$ and $\cos\alpha_{k',m'}^{(i)}<0$, we have $b_{k,m}(\alpha_{k,m}^{(i)})-b_{k,m}(\alpha_{k',m'}^{(i)})>0$.
As $b_{k,m}(\theta)$ is an increasing function, we have $\alpha_{k,m}^{(i)}>\alpha_{k',m'}^{(i)}$.
\end{proof}

\begin{rmk}\label{1st z}
In view of \lemref{1st zero}, we note that for any $k'>k\ge50$ be even integers 
and $m'\le m\le -1$, we have $\alpha_{k',m'}^{(i)}\le \max\{\alpha_{k,m}^{(i)},\alpha_{k+2,m}^{(i)}\}$.
\end{rmk}
Using this, we prove the following.
\begin{lem}\label{ulb}
For $m<0$, we have 
$$
\frac{\pi}{2}+\frac{\pi}{k-4m\pi}\left(\frac{r_k}{2}+1\right)<{\alpha}_{k,m}^{(1)}<\frac{\pi}{2}+\frac{\pi}{k-4m\pi\left(1-\frac{1}{2}\left(\frac{\pi(2+r_k)}{2(k-3.96m\pi)}\right)^2\right)}\left(\frac{r_k}{2}+1\right).
$$
\end{lem}
\begin{proof}
Using \eqref{bkma}, we have
\begin{equation}\label{bkm}
b_{k,m}({\alpha}_{k,m}^{(1)})=\frac{k\pi}{4}+\frac{r_k\pi}{4}+\frac{\pi}{2}.
\end{equation}
Let $L_{k,m}(\theta)={k\theta}/{2}-2\pi m\left(\theta-{\pi}/{2}\right)$. Then $L_{k,m}(\theta)-b_{k,m}(\theta)$ is an 
increasing function of $\theta\in I^{\circ}$. This is because its derivative $-2\pi m+2\pi m\sin\theta$ is positive for 
$\theta\in I^{\circ}$. Since $L_{k,m}(\pi/2)=b_{k,m}(\pi/2)$, we have $L_{k,m}(\theta)-b_{k,m}(\theta)>0$, for all $\theta\in I^{\circ}$.
If $\mu_1$ is a point such that 
$$
L_{k,m}(\mu_1)=\frac{k\pi}{4}+\frac{r_k\pi}{4}+\frac{\pi}{2},
$$
then $\mu_1 < {\alpha}_{k,m}^{(1)}$. This is because $b_{k,m}(\theta)$ is an increasing function for $\theta\in I^{\circ}$ and
\begin{align*}
    b_{k,m}({\alpha}_{k,m}^{(1)})-b_{k,m}(\mu_1)
&=b_{k,m}({\alpha}_{k,m}^{(1)})-L_{k,m}(\mu_1)+L_{k,m}(\mu_1)-b_{k,m}(\mu_1)\\
&=L_{k,m}(\mu_1)-b_{k,m}(\mu_1)>0,
\end{align*}
$$
\mu_1 = \frac{\pi}{2}+\frac{\pi}{k-4m\pi}\left(\frac{r_k}{2}+1\right)<{\alpha}_{k,m}^{(1)}.
$$
For an even integer $k\ge 50$ and an integer $m\le -1$, using \rmkref{1st z}, we have $\alpha_{k,m}^{(1)}\le \max\{\alpha_{50,-1}^{(1)},\alpha_{52,-1}^{(1)}\}
\le 1.68$. Let $I_1:=(\pi/2,1.68)$. Then for $k\ge 50$ even and $m\le-1$, we have $\alpha_{k,m}^{(1)}\in I_1$. For $\theta\in I_1$, we define $A_{k,m}(\theta)={k\theta}/{2}-1.98m\pi\left(\theta-{\pi}/{2}\right)$. Then the function 
$b_{k,m}(\theta)-A_{k,m}(\theta)$ is positive for $\theta\in I_1$ as $2\sin(\theta)>1.98$ for 
$\theta\in I_1$. Let $\mu_2$ be such that 
$$
A_{k,m}(\mu_2)=\frac{k\pi}{4}+\frac{r_k\pi}{4}+\frac{\pi}{2}.
$$
Then 
$$
\mu_2=\frac{\pi}{2}+\frac{\pi}{k-3.96m\pi}\left(\frac{r_k}{2}+1\right)\le \frac{\pi}{2}+\frac{2\pi}{50+3.96\pi}<1.68.
$$
Analysing similarly we have $\mu_2 > {\alpha}_{k,m}^{(1)}$. 

The upper bound obtained this way is quite far from the lower bound, for the large values of $m$, and is not 
useful for the computation of $\lambda_{l,m}$. Therefore, we obtain another upper bound, which is much closer to the lower bound. 
For that, we already have
$$
{\alpha}_{k,m}^{(1)}\in I_2:=\left(\frac{\pi}{2},\frac{\pi}{2}+\frac{\pi}{k-3.96m\pi}\left(\frac{r_k}{2}+1\right)\right).
$$
In this interval, we consider the function. 
$$
\tilde A_{k,m}(\theta)=\frac{k\theta}{2}-2m\pi\left(1-\frac{1}{2}\left(\frac{\pi(2+r_k)}{2(k-3.96m\pi)}\right)^2\right)\left(\theta-\frac{\pi}{2}\right).
$$
Next we prove that $b_{k,m}(\theta)-\tilde A_{k,m}(\theta)>0$ for $\theta\in I_2$. We consider
$$
f_{k,m}(\theta)=b_{k,m}(\theta)-\tilde A_{k,m}(\theta)=2\pi m\cos\theta+2m\pi\left(1-\frac{1}{2}\left(\frac{\pi(2+r_k)}{2(k-3.96m\pi)}\right)^2\right)\left(\theta-\frac{\pi}{2}\right).
$$
Now
$$
f_{k,m}'(\theta)=-2m\pi \sin\theta+2m\pi\left(1-\frac{1}{2}\left(\frac{\pi(2+r_k)}{2(k-3.96m\pi)}\right)^2\right).
$$
For $\theta\in I_2$ and $m<0$, the function $-2\pi m\sin\theta$ is positive, and it is a decreasing function. Therefore, a lower bound is as follows:
\begin{align*}
-2\pi m\sin\theta&>-2\pi m\sin\left(\frac{\pi}{2}+\frac{\pi(r_k+2)}{2(k-3.96m\pi)}\right)\\
&=-2\pi m\cos \left(\frac{\pi(r_k+2)}{2(k-3.96m\pi)}\right)\\
&>-2\pi m \left(1-\frac{1}{2}\left(\frac{\pi(r_k+2)}{2(k-3.96m\pi)}\right)^2\right).
\end{align*}
Hence $f_{k,m}(\theta)$ is an increasing function for $\theta\in I_2$ and $f_{k,m}(\pi/2)=0$. So $b_{k,m}(\theta)-
\tilde A_{k,m}(\theta)>0$ for $\theta\in I_2$. Now if $\mu_3$ be such that 
$$
\tilde A_{k,m}(\mu_3)=\frac{k\pi}{4}+\frac{r_k\pi}{4}+\frac{\pi}{2},
$$
then 
$$
\mu_3=\frac{\pi}{2}+\frac{\pi}{k-4m\pi\left(1-\frac{1}{2}\left(\frac{\pi(2+r_k)}{2(k-3.96m\pi)}\right)^2\right)}\left(\frac{r_k}{2}+1\right).
$$
For $k\ge 50$, even and $m\le-1$, 
$$
\frac{1}{2}\left(\frac{\pi(2+r_k)}{2(k-3.96m\pi)}\right)^2<0.01.
$$
Therefore, 
$$
k-4m\pi\left(1-\frac{1}{2}\left(\frac{\pi(2+r_k)}{2(k-3.96m\pi)}\right)^2\right)\ge k-3.96m\pi.
$$
Hence $\mu_3\in I_2$, and following analysis similar to earlier we get
$$
{\alpha}_{k,m}^{(1)}<\mu_3.
$$.
\end{proof}

Next, we prove the following lemma.
\begin{lem}\label{cons}
Let $m<0$ be such that  $0.06\le \{-2m\pi\}\le 0.99$. Then for any $l\ge 50$ even, there are exactly $4$ 
consecutive integers in the following interval
$$J_{l,m}:=\left[-2m\pi-4,-2m\pi+m\pi\left(\frac{2\pi}{l-3.96m\pi}\right)^2\right].$$
\end{lem}
\begin{proof}
First we prove that for $l\ge 50$ even and $m\le-1$, we have
$$
0<-m\pi\left(\frac{2\pi}{l-3.96m\pi}\right)^2<0.05.
$$
For any $l\ge 50$, then function in middle is not monotonic in $m$, 
we consider the function
$$
h_{l}(x):=-x\pi\left(\frac{2\pi}{l-3.96x\pi}\right)^2
$$
For $x<0$, this function has a unique extremum at $x=-l/(3.96\pi)$, and as 
$x\to \infty$, $h_l(x)\to 0$, therefore
$$
h_{l}(x)\le\frac{\pi^2}{3.96l}.
$$
For $l\ge 50$ even $h_l(x)< 0.05$ and therefore for any $m\le-1$, we have
\begin{equation}\label{0.09}
0<-m\pi\left(\frac{2\pi}{l-3.96m\pi}\right)^2<0.05.
\end{equation}
So, the condition $\{-2m\pi\}\ge 0.06$ ensures that there are exactly $4$ consecutive integers 
in the given range, which proves the lemma.
\end{proof}

In view of \lemref{cons}, for any integer $l>50$, even and an integer $m<0$ satisfying $l\equiv2\pmod 4$ and
$0.06\le\{-2m\pi\}\le0.99$, there exists an integer $\lambda_{l,m}'\in J_{l,m}$ such that 
$l/2-\lambda_{l,m}'\equiv0\pmod4$. We define $\lambda_{l,m}$ as follows:
$$
\lambda_{l,m}=
\begin{cases}
\lambda_{l,m}', & \ \text{if}\   l/2-\lambda_{l,m}'\ge52\\
l/2-52, & \ \text{if}\ l/2-\lambda_{l,m}'<52.
\end{cases}
$$

\begin{lem}\label{k<lam}
Let $l\ge 50$ even be such that $r_l=2$. Then we have 
$$
\alpha_{(l-2\lambda_{l,m})/2,m}^{(1)}<\alpha_{l,m}^{(1)}.
$$
\end{lem}

\begin{proof} Consider the following
\begin{align*}
&b_{(l-2\lambda_{l,m})/2}(\alpha_{l,m}^{(1)})-b_{(l-2\lambda_{l,m})/2}(\alpha_{(l-2\lambda_{l,m})/2,m}^{(1)})\\
&=b_{l}(\alpha_{l,m}^{(1)})-\frac{(l+2\lambda_{l,m})\alpha_{l,m}^{(1)}}{4}
-b_{(l-2\lambda_{l,m})/2}(\alpha_{(l-2\lambda_{l,m})/2,m}^{(1)}).
\end{align*}
Using \eqref{bkma}, we have
\begin{align*}
&b_{(l-2\lambda_{l,m})/2}(\alpha_{l,m}^{(1)})-b_{(l-2\lambda_{l,m})/2}(\alpha_{(l-2\lambda_{l,m})/2,m}^{(1)})\\&=\frac{l\pi}{4}+\pi-\frac{(l-2\lambda_{l,m})\pi}{8}-\frac{\pi}{2}-\frac{(l+2\lambda_{l,m})\alpha_{l,m}^{(1)}}{4}\\
&=\frac{(l+2\lambda_{l,m})}{4}\left(\frac{\pi}{2}-\alpha_{l,m}^{(1)}\right)+\frac{\pi}{2}.
\end{align*}
Now using \lemref{ulb}, we get
\begin{align}\label{k>l/2}
b_{(l-2\lambda_{l,m})/2}(\alpha_{l,m}^{(1)})-b_{(l-2\lambda_{l,m})/2}(\alpha_{(l-2\lambda_{l,m})/2,m})
>\frac{\pi}{2}-\frac{l+2\lambda_{l,m}}{4}\left(\frac{2\pi}{l-4m\pi\left(1-\frac{1}{2}\left(\frac{2\pi}{l-3.96m\pi}\right)^2\right)}\right).
\end{align}
Since $\lambda_{l,m}\le \lambda_{l,m}' \in J_{l,m}$, we have
$$
l+2\lambda_{l,m}<l-4m\pi\left(1-\frac{1}{2}\left(\frac{2\pi}{l-3.96m\pi}\right)^2\right).
$$
Hence $b_{(l-2\lambda_{l,m})/2}(\alpha_{l,m}^{(1)})>b_{(l-2\lambda_{l,m})/2}(\alpha_{(l-2\lambda_{l,m})/2,m}^{(1)})
$. As $b_{(l-2\lambda_{l,m})/2}(\theta)$ is an increasing function we have $\alpha_{l,m}^{(1)}>\alpha_{(l-2\lambda_{l,m})/2,m}^{(1)}$.

\end{proof}

\begin{lem}\label{k>lam}
Let $l\ge 50$ even be such that $r_l=2$, and $l/2-\lambda_{l,m}-4\ge50$.
Then we have 
$$
\alpha_{l,m}^{(1)}<\alpha_{(l-2\lambda_{l,m}-8)/2,m}^{(1)}.
$$
\end{lem}
\begin{proof}

Now we consider
\begin{align*}
    &b_{l,m}(\alpha_{(l-2\lambda_{l,m}-8)/2,m}^{(1)})-b_{l,m}(\alpha_{l,m}^{(1)})\\
    &=b_{(l-2\lambda_{l,m}-8)/2,m}(\alpha_{(l-2\lambda_{l,m}-8)/2,m}^{(1)})+\frac{(l+2\lambda_{l,m}+8)\alpha_{(l-2\lambda_{l,m}-8)/2,m}}{4}-b_{l,m}(\alpha_{l,m}^{(1)})
    \end{align*}
    Using \eqref{bkma}, we have
    $$
   b_{l,m}(\alpha_{(l-2\lambda_{l,m}-8)/2,m}^{(1)})-b_{l,m}(\alpha_{l,m}^{(1)}) =\frac{l+2\lambda_{l,m}+8}{4}\left(\alpha_{(l-2\lambda_{l,m}-8)/2,m}^{(1)}-\frac{\pi}{2}\right)
-\frac{\pi}{2}.
$$
Now using \lemref{ulb}, we get
\begin{equation}\label{k<l/2}
b_{l,m}(\alpha_{(l-2\lambda_{l,m}-8)/2,m}^{(1)})-b_{l,m}(\alpha_{l,m}^{(1)})>\frac{l+2\lambda_{l,m}+8}{4}\left(\frac{2\pi}{l-2\lambda_{l,m}-8-8m\pi}\right)-\frac{\pi}{2}.
\end{equation}
Since $l/2-\lambda_{l,m}\ge54$, we have $\lambda_{l,m}=\lambda_{l,m}'\in J_{l,m}$, and $\lambda_{l,m}$ is an integer, we have
$$
\lambda_{l,m}>-2m\pi-4.
$$
Using this
$$
l+2\lambda_{l,m}+8>l-2\lambda_{l,m}-8-8m\pi.
$$
This gives us $b_{l,m}(\alpha_{(l-2\lambda_{l,m}-8)/2,m}^{(1)})>b_{l,m}(\alpha_{l,m}^{(1)})$ or equivalently,
$\alpha_{(l-2\lambda_{l,m}-8)/2,m}^{(1)}>\alpha_{l,m}^{(1)}$.
\end{proof}

\section{Proof of \thmref{main}}\label{pmain}
In this section, we prove \thmref{main} by proving the ordering of the zeros mentioned at the end of \S\ref{struc}.
Let $\delta_{k,m,i}=|\alpha_{k,m}^{(i)}-\tilde\alpha_{k,m}^{(i)}|$. If for some $1\le i\le m(l)-m$ and $1\le j\le m(k)-m$, we have $\alpha_{l,m}^{(i)}<\alpha_{k,m}^{(j)}$ (or $\alpha_{l,m}^{(i)}>\alpha_{k,m}^{(j)}$)  and $|\alpha_{k,m}^{(j)}-\alpha_{l,m}^{(i)}|<\delta_{k,m,j}+\delta_{l,m,i}$, then we have $\tilde\alpha_{l,m}^{(i)}<\tilde\alpha_{k,m}^{(j)}$ (or 
$\tilde\alpha_{l,m}^{(i)}>\tilde\alpha_{k,m}^{(j)}$), respectively. Therefore, to get the desired ordering among the zeros of 
$H_{k,m}(\theta)$ and $H_{l,m}(\theta)$, we need the following:
\begin{itemize}
\item An upper bound of $\delta_{k,m,i}$ for $i=1,2$.
\item The desired ordering among the corresponding zeros of $\cos b_{k,m}(\theta)$ and $\cos b_{l,m}(\theta)$ and a lower bound 
for their distance.
\end{itemize}

\subsection{Displacement calculation} The objective of this section is to compute  
upper bounds for $\delta_{k,m,1}$ and $\delta_{k,m,2}$. We start with the following two remarks.
\begin{rmk}\rm\label{eta}
Let $\kappa$ denote the minimum distance between two distinct zeros of $\cos b_{k,m}(\theta)$. Suppose $i\in\{1,2,\ldots, m(k)-m\}$ is such that
$\kappa={\alpha}_{k,m}^{(i+1)}-\alpha_{k,m}^{(i)}$. By the mean value theorem, there exists
$\xi_i\in I^{\circ}$ such that
$$
b_{k,m}'(\xi_i)=\frac{b_{k,m}({\alpha}_{k,m}^{(i+1)})-b_{k,m}(\alpha_{k,m}^{(i)})}{{\alpha}_{k,m}^{(i+1)}-\alpha_{k,m}^{(i)}}=\frac{\pi}{{\alpha}_{k,m}^{(i+1)}-\alpha_{k,m}^{(i)}}.
$$
Now for $\theta\in I^{\circ}$, we have 
\begin{equation}\label{bk'}
b_{k,m}'(\theta)\le \frac{k}{2}-2m\pi,
\end{equation}
and hence
$$
\kappa=\frac{\pi}{b_{k,m}'(\xi_i)}>\frac{2\pi}{(k-4m\pi)}.
$$
\end{rmk}

\begin{rmk}\rm\label{eta'}
Let $\kappa'=2\pi/3-\alpha_{k,m}^{(m(k)-m)}$ and $\kappa{''}=\alpha_{k,m}^{(1)}-\pi/2$. Using \eqref{bkma}, we have
$$
b_{k,m}({2\pi}/{3})-b_{k,m}(\alpha_{k,m}^{(m(k)-m)})\ge\frac{\pi}{2},\  \text{ and }
 \ b_{k,m}(\alpha_{k,m}^{(1)})-b_{k,m}({\pi}/{2})\ge\frac{\pi}{2}.$$ 
Thus, by using the mean value theorem as in \rmkref{eta}, we get
$$
\kappa{'}, \kappa''>\frac{\pi}{k-4m\pi}.
$$
\end{rmk}
In the following lemma, we derive an upper bound for $\delta_{k,m,1}$.
\begin{lem}\label{disp}
For $k\ge 50$ even and $m\le -1$, we have 
$$
\delta_{k,m,1}\le \frac{0.001\pi e^{0.87m\pi}}{k(k-2\sqrt3m\pi)}.
$$
\end{lem}
\begin{proof}
We define $\gamma_{k,m,1}:= \frac{0.001\pi}{(k-2\sqrt3m\pi)(k-4m\pi)}$. Then in view of \rmkref{eta'}, we have
$$
\frac{\pi}{2}<\alpha_{k,m}^{(1)}\pm\gamma_{k,m,1}< \frac{2\pi}{3}.
$$
Also in view of \rmkref{eta}, $\cos b _{k,m}(\alpha_{k,m}^{(1)}+\gamma_{k,m,1})\cos b _{k,m}(\alpha_{k,m}^{(1)}-\gamma_{k,m,1})<0$.
Therefore, to prove the lemma, it is enough to prove that the sign of $H_{k,m}(\alpha_{k,m}^{(1)}\pm\gamma_{k,m,1})$ is the same as the sign of 
$\cos b_{k,m}(\alpha_{k,m}^{(1)}\pm\gamma_{k,m,1})$. To show this, in view of 
\eqref{hkm}, we need to prove that
$$
|2\cos b_{k,m}(\alpha_{k,m}^{(1)}\pm\gamma_{k,m,1})|>|T_{k,m}(\alpha_{k,m}^{(1)}\pm\gamma_{k,m,1})|.
$$
Following steps similar to \cite[\S3]{KS1}, we have 
$$
|2\cos b_{k,m}(\alpha_{k,m}^{(1)}\pm\gamma_{k,m,1})|>\frac{6}{\pi}\left(\frac{k\gamma_{k,m,1}}{2}-2\pi m\gamma_{k,m,1}\sin\xi\right),
$$
for some $\xi\in (\pi/2,2\pi/3)$. Using $-2\pi m\gamma_{k,m,1}\sin\xi\ge -\sqrt 3\pi m\gamma_{k,m,1}$, we have
$$
|2\cos b_{k,m}(\alpha_{k,m}^{(1)}\pm\gamma_{k,m,1})|>\frac{3\gamma_{k,m,1}}{\pi}(k-2\sqrt 3m\pi)=\frac{0.003 e^{0.87m\pi}}{k}.
$$
Since $P_{k,m}(\theta)$ is an increasing function, using \eqref{qkm}, \eqref{rkm}, we have
$$
|T_{k,m}(\alpha_{k,m}^{(1)}\pm\gamma_{k,m,1})|\le \frac{e^{m\pi(2\sin(\alpha_{k,m}^{(1)}+\gamma_{k,m,1})-\tan\frac {\alpha_{k,m}^{(1)}+\gamma_{k,m,1}}{2})}}{(2\cos\frac {\alpha_{k,m}^{(1)}+\gamma_{k,m,1}}{2})^k}+
\frac{e^{m\pi}}{2^{k/2}}+\frac{0.359e^{1.039m\pi}}{2^{k/2}}.
$$
Using \rmkref{1st z} for $k\ge 50$ even and $m\le -1$, we have
$$
\alpha_{k,m}^{(1)}+\gamma_{k,m,1}\le \min\{\alpha_{50,-1}^{(1)},\alpha_{52,-1}^{(1)}\}+\gamma_{k,m,1}\le 1.68.
$$
And hence 
$$
|T_{k,m}(\alpha_{k,m}^{(1)}\pm\gamma_{k,m,1})|\le 
\frac{e^{m\pi(2\sin1.68-\tan\frac {1.68}{2})}}{(2\cos\frac {1.68}{2})^k}+\frac{e^{m\pi}}{2^{k/2}}+
\frac{0.359e^{1.039m\pi}}{2^{k/2}}.
$$
Since $2\sin(1.68)-\tan {(0.84)}\ge 0.87$, for $k\ge 50$ even and $m\le -1$, it can be checked that 
$$
\frac{e^{m\pi(2\sin(1.68)-\tan (0.84))}}{(2\cos\frac {1.68}{2})^k}+\frac{e^{m\pi}}{2^{k/2}}+
\frac{0.359e^{1.039m\pi}}{2^{k/2}}<\frac{0.003e^{0.87m\pi}}{k}.
$$
This proves that
$$
H_{k,m}(\alpha_{k,m}^{(1)}+\gamma_{k,m,1})H_{k,m}(\alpha_{k,m}^{(1)}-\gamma_{k,m,1})<0.
$$
\end{proof}
In the following lemma, we compute an upper bound for $\delta_{k,m,2}$.
\begin{lem}\label{disp2}
For $k\ge 50$ even and $m\le -1$, we have 
$$
\delta_{k,m,2}\le\frac{0.1\pi}{k-2\sqrt3m\pi}.
$$
\end{lem}
\begin{proof}
Let $\gamma_{k,m,2}:=\frac{0.1\pi}{k-2\sqrt3m\pi}$. An analysis similar to the proof of \lemref{disp}
gives us
$$
|2\cos b_{k,m}(\alpha_{k,m}^{(1)}\pm\gamma_{k,m,2})|>\frac{3\gamma_{k,m,2}}{\pi}(k-2\sqrt 3m\pi)={0.3}.
$$
In view of \eqref{hkm}, to prove that the sign of $H_{k,m}(\alpha_{k,m}^{(1)}\pm\gamma_{k,m,2})$ is same as the sign of $\cos b_{k,m}(\alpha_{k,m}^{(1)}\pm\gamma_{k,m,2})$, we need to show that
$$
|T_{k,m}(\alpha_{k,m}^{(2)}\pm\gamma_{k,m,2})|\le{0.3}.
$$
For $k\ge 50$ even and $m\le -1$, using \rmkref{1st z}, we have
$$
\alpha_{k,m}^{(2)}+\gamma_{k,m,2}\le \min\{\alpha_{50,-1}^{(2)},\alpha_{52,-1}^{(2)}\}+\gamma_{k,m,2}\le 1.78.
$$
Therefore for $k\ge 50$ even and $m\le -1$ we have 
$$
|T_{k,m}(\alpha_{k,m}^{(2)}\pm\gamma_{k,m,2})|\le 
\frac{e^{m\pi(2\sin1.78-\tan\frac {1.78}{2})}}{(2\cos\frac {1.78}{2})^k}+\frac{e^{m\pi}}{2^{k/2}}+
\frac{0.359e^{1.039m\pi}}{2^{k/2}}\le 0.3.
$$
This completes the proof.
\end{proof}

\subsection{Ordering of the zeros of $H_{k,m}(\theta)$ and $H_{l,m}(\theta)$}
The key purpose of this section
is to derive the desired ordering 
among the zeros of $\cos b_{k,m}(\theta)$ and $\cos b_{l,m}(\theta)$ first, and a lower bound for their distance. The distance computation between any two zeros of $\cos b_{k,m}(\theta)$ and $\cos b_{l,m}(\theta)$ in the case $(r_k,r_l)=(2,0)$ is a bit tricky and requires  $-2m\pi$ to be away from its nearest integers. Therefore, the condition 
on the fractional part of $-2m\pi$ is important in calculating the distance.
The desired ordering among the zeros of $H_{k,m}(\theta)$ and 
$H_{l,m}(\theta)$ can be obtained as a corollary to all these computations. We proceed by dividing into $3$ cases namely, $(r_k,r_l)=(2,0)$, $r_k=r_l$ and $(r_k,r_l)=(0,2)$.

\subsubsection{The case $(r_k,r_l)=(2,0)$} In this subsection we prove that for $l>k\ge50$ even integers such that 
$(r_k,r_l)=(2,0)$, we have $\tilde\alpha_{l,m}^{(1)}<\tilde\alpha_{k,m}^{(1)}$. For that, we establish the following lemma.
\begin{lem}\label{l(2,0)}
For $l>k\ge50$ even integers such that $(r_k,r_l)=(2,0)$, we have $\alpha_{l,m}^{(1)}<\alpha_{k,m}^{(1)}$. Moreover,
$$
\alpha_{k,m}^{(1)}-\alpha_{l,m}^{(1)}>\frac{\pi}{k-4m\pi}.
$$
\begin{proof}
In view of \eqref{bkm}, we have
$$
b_{k,m}(\alpha_{k,m}^{(1)})=\frac{k\pi}{4}+\pi,  \ \text{and} \ b_{l,m}(\alpha_{l,m}^{(1)})=\frac{l\pi}{4}+\frac{\pi}{2}.
$$
We consider the following
\begin{align*}
b_{k,m}(\alpha_{k,m}^{(1)})-b_{k,m}(\alpha_{l,m}^{(1)})&=b_{k,m}(\alpha_{k,m}^{(1)})-b_{l,m}(\alpha_{l,m}^{(1)})+\frac{(l-k)\alpha_{l,m}^{(1)}}{2}\\
&=\frac{k\pi}{4}+\pi-\frac{l\pi}{4}-\frac{\pi}{2}+\frac{(l-k)\alpha_{l,m}^{(1)}}{2}\\
&=\frac{\pi}{2}+\frac{(l-k)}{2}\left(\alpha_{l,m}^{(1)}-\frac{\pi}{2}\right)>\frac{\pi}{2}.
\end{align*}
Since $b_{k,m}(\theta)$ is an increasing function, we have $\alpha_{k,m}^{(1)}>\alpha_{l,m}^{(1)}$. 
Moreover by using the mean value theorem for some $\xi\in (\pi/2,2\pi/3)$, we have
$$
\alpha_{k,m}^{(1)}-\alpha_{l,m}^{(1)}>\frac{\pi}{2b_{k,m}'(\xi)}>\frac{\pi}{k-4m\pi}.
$$
\end{proof}
\end{lem}

\begin{cor}\label{c2,0}
For $l>k\ge50$ even integers such that $(r_k,r_l)=(2,0)$, we have $\tilde\alpha_{l,m}^{(1)}<\tilde\alpha_{k,m}^{(1)}$.
\end{cor}
\begin{proof}
Using \lemref{disp} and \lemref{l(2,0)}, we have $\delta_{k,m,1}+\delta_{l,m,1}<\alpha_{k,m}^{(1)}-\alpha_{l,m}^{(1)}$, and therefore $\tilde\alpha_{l,m}^{(1)}<\tilde\alpha_{k,m}^{(1)}$.
\end{proof}

\subsubsection{The case $r_k=r_l$} In this subsection we prove that 
for $l>k\ge 50$ even integers such that $r_k=r_l$, we have $\tilde\alpha_{l,m}^{(1)}<\tilde\alpha_{k,m}^{(1)}$.
For that, we first prove the following lemma.

\begin{lem}\label{lequal}
For $l>k\ge50$ even integers such that $r_k=r_l$, we have $\alpha_{l,m}^{(1)}<\alpha_{k,m}^{(1)}$. Moreover,
$$
\alpha_{k,m}^{(1)}-\alpha_{l,m}^{(1)}>\frac{(l-k)\pi}{2(k-4m\pi)}.
$$
\end{lem}
\begin{proof}
We consider the following:
\begin{align*}
b_{k,m}(\alpha_{k,m}^{(1)})-b_{k,m}(\alpha_{l,m}^{(1)})&=b_{k,m}(\alpha_{k,m}^{(1)})-b_{l,m}(\alpha_{l,m}^{(1)})+\frac{(l-k)\alpha_{l,m}^{(1)}}{2}\\
&=\frac{k\pi}{4}-\frac{l\pi}{4}+\frac{(l-k)\alpha_{l,m}^{(1)}}{2}\\
&=\frac{(l-k)}{2}\left(\alpha_{l,m}^{(1)}-\frac{\pi}{2}\right).
\end{align*}
Now, by using \lemref{ulb}, we get
$$
b_{k,m}(\alpha_{k,m}^{(1)})-b_{k,m}(\alpha_{l,m}^{(1)})>\frac{(l-k)\pi}{2(l-4m\pi)}.
$$
Therefore, by using the mean value theorem, we get
$$
\alpha_{k,m}^{(1)}-\alpha_{l,m}^{(1)}>\frac{(l-k)\pi}{(l-4m\pi)(k-4m\pi)}.
$$
\end{proof}

\begin{cor}\label{cequal}
For $l>k\ge 50$ even integers such that $r_k=r_l$, we have $\tilde\alpha_{l,m}^{(1)}<\tilde\alpha_{k,m}^{(1)}$.
\end{cor}
\begin{proof}
In view of \lemref{disp} and \lemref{lequal}, it is sufficient to prove that
$$
\frac{(l-k)\pi}{(l-4m\pi)(k-4m\pi)}>\frac{0.001\pi e^{0.87m\pi}}{k(k-2\sqrt3m\pi)}+\frac{0.001\pi e^{0.87m\pi}}{l(l-2\sqrt3m\pi)}.
$$
As $r_k=r_l$, we have $l-k\ge 4$.
Now for $2k>l$, we have 
$$2(k-2\sqrt3m\pi)=2k-4\sqrt3m\pi>l-2\sqrt3m\pi.$$ and therefore,
$$
\frac{4}{l(l-2\sqrt3m\pi)}>\frac{1}{k(k-2\sqrt3m\pi)}.
$$
Hence for $2k>l$,
$$
\frac{0.001\pi e^{0.87m\pi}}{k(k-2\sqrt3m\pi)}+\frac{0.001\pi e^{0.87m\pi}}{l(l-2\sqrt3m\pi)}\le\frac{0.005\pi e^{0.87m\pi}}{l(l-2\sqrt3m\pi)}.
$$
Now for $l\ge 50$ and $m\le-1$, we have
$$
\frac{\pi(l-k)}{(l-4m\pi)(k-4m\pi)}>\frac{4\pi}{(l-4m\pi)^2}>\frac{0.005\pi e^{0.87m\pi}}{l(l-2\sqrt3m\pi)}.
$$
For $2k\le l$, we have  $l-k\ge{l}/{2}$, and therefore we have
$$
\frac{\pi(l-k)}{(l-4m\pi)(k-4m\pi)}\ge\frac{ l\pi}{2(l-4m\pi)(k-4m\pi)}> \frac{\pi}{2(k-4m\pi)(1-4m\pi)}.
$$
Also for $2k\le l$,
$$
\frac{1}{l(l-2\sqrt3m\pi)}<\frac{1}{k(k-2\sqrt3m\pi)}.
$$
Now for $k\ge 50$, and $m\le-1$, we have
$$
\frac{\pi}{2(k-4m\pi)(1-4m\pi)}>\frac{0.002\pi e^{0.87m\pi}}{k(k-2\sqrt3m\pi)}.
$$
This completes the proof.
\end{proof}

\subsubsection{The case $(r_k,r_l)=(2,0)$}
For any $50\le k<l$ even integers such that $(r_k,r_l)=(0,2)$, using \lemref{k<lam}, \lemref{k>lam} and \lemref{l(2,0)}, we have
\begin{align*}
&\alpha_{l,m}^{(1)}<\alpha_{(l-2\lambda_{l,m}-8)/2,m}^{(1)}\le\alpha_{k,m}^{(1)} \ \ \text{for}\ \
k\le (l-2\lambda_{l,m}-8)/2, \\
&\text{and}\ \alpha_{k,m}^{(1)}\le \alpha_{(l-2\lambda_{l,m})/2,m}^{(1)}<\alpha_{l,1}^{(1)}
\ \ \text{for}\ \ k\ge (l-2\lambda_{l,m})/2.
\end{align*}
Therefore, in view of \corref{c2,0}, to get the same ordering among the zeros of 
$H_{k,m}(\theta)$ and $H_{l,m}(\theta)$, we only need to prove the following:
$$
    \tilde\alpha_{(l-2\lambda_{l,m})/2,m}^{(1)}<\tilde\alpha_{l,m}^{(1)}, \ \ \text{for}\ \ l/2-\lambda_{l,m}\ge 50,\ \  \text{and}
    $$
    $$
\tilde\alpha_{(l-2\lambda_{l,m}-8)/2,m}^{(1)}>\tilde\alpha_{l,m}^{(1)},    \ \ \text{for}\ \ l/2-\lambda_{l,m}-4\ge50.
$$

\begin{lem}\label{lk<l/2}
Let $l\ge50$ be an even integer such that $r_l=2$. For an integer $m\le -1$ such that
$l/2-\lambda_{l,m}-4\ge 50$, we have
$$
\alpha_{(l-2\lambda_{l,m}-8)/2,m}^{(1)}-\alpha_{l,m}^{(1)}>\frac{0.04\pi}{(l-4m\pi-0.02)(l-4m\pi)}.
$$
\end{lem}

\begin{proof}
In view of \eqref{k<l/2}, we have
\begin{equation*}
b_{l,m}(\alpha_{(l-2\lambda_{l,m}-8)/2,m}^{(1)})-b_{l,m}(\alpha_{l,m}^{(1)})>\frac{l+2\lambda_{l,m}+8}{4}\left(\frac{2\pi}{l-2\lambda_{l,m}-8-8m\pi}\right)-\frac{\pi}{2}.
\end{equation*}
As $l/2-\lambda_{l,m}\ge 54$, so $\lambda_{l,m}=\lambda_{l,m}'\in J_{l,m}$, using $\{-2m\pi\}\le 0.99$, we get that
$$
\lambda_{l,m}\ge -2m\pi-4+0.01.
$$
And hence
$$
\frac{l+2\lambda_{l,m}+8}{l-2\lambda_{l,m}-8-8m\pi}\ge \frac{l-4m\pi+0.02}{l-4m\pi-0.02}=1+\frac{0.04}{l-4m\pi-0.02}.
$$
Therefore,
$$
b_{l,m}(\alpha_{(l-2\lambda_{l,m}-8)/2,m}^{(1)})-b_{l,m}(\alpha_{l,m}^{(1)})>\frac{0.02\pi}{l-4m\pi-0.02}.
$$
Since $\alpha_{(l-2\lambda_{l,m}-8)/2,m}^{(1)}-\alpha_{l,m}^{(1)}\neq0$, by using the mean value theorem we
have
$$
\alpha_{(l-2\lambda_{l,m}-8)/2,m}^{(1)}-\alpha_{l,m}^{(1)}>\frac{0.04\pi}{(l-4m\pi-0.02)(l-4m\pi)}>\frac{0.04\pi}{(l-4m\pi)^2}.
$$
\end{proof}

\begin{cor}\label{ck<l/2}
Let $l\ge50$ be an even integer such that $r_l=2$. For an integer $m\le -1$ such that
$l/2-\lambda_{l,m}-4\ge 50$ we have
$\tilde\alpha_{(l-2\lambda_{l,m}-8)/2,m}^{(1)}>\tilde\alpha_{l,m}^{(1)}$.
\end{cor}
\begin{proof}
%
%
Using \lemref{disp} for $l$ and $m$ satisfying the given conditions we have
 $$
 \delta_{l,m,1}\le \frac{0.001\pi e^{0.87m\pi}}{l(l-2\sqrt3m\pi)}\le \frac{0.01\pi}{(l-4m\pi)^2}.
 $$
In view of \lemref{disp} and \lemref{lk<l/2}, we need to prove that 
  $$
  \delta_{(l-2\lambda_{l,m}-8)/2,m,1}\le \frac{0.004\pi e^{0.87m\pi}}{(l-2\lambda_{l,m}-8)(l-2\lambda_{l,m}-8-4\sqrt3 m\pi)}<\frac{0.03\pi}{(l-4m\pi)^2}.
 $$
 Since $l/2-\lambda_{l,m}-4\ge50$, we have
 $ \lambda_{l,m}=\lambda_{l,m}'\in[-2m\pi-4,-2m\pi]$. Using this we have
\begin{align*}
 (l-2\lambda_{l,m}-8)(l-2\lambda_{l,m}-8-4\sqrt3 m\pi)&\ge (l+4m\pi-8)(l+4m\pi-8-4\sqrt3m\pi)\\&\ge l(l+4m\pi-8).
\end{align*}
So,
 $$
  \delta_{(l-2\lambda_{l,m}-8)/2,m,1}\le \frac{0.004\pi e^{0.87m\pi}}{l(l+4m\pi-8)}.
 $$
Therefore,
 $$
 \frac{0.004\pi e^{0.87m\pi}}{l(l+4m\pi-8)}\le \frac{0.03\pi}{(l-4m\pi)^2}\ \  \text{  if and only if  }\  \frac{e^{0.87m\pi}(l-4m\pi)^2}{l(l+4m\pi-8)}<\frac{30}{4}.
 $$
 First we assume that $l>-8m\pi$, then we have $-4m\pi<l/2$. So,
 $$
 \frac{e^{0.87m\pi}(l-4m\pi)^2}{l(l+4m\pi)}\le \frac{e^{0.87m\pi}(l+\frac{l}{2})^2}{l(l-\frac{l}{2}-8)}=\frac{9le^{0.87m\pi}}{2(l-16)}<\frac{30}{4}.
 $$
Next, we assume that $l<-8m\pi$. Since $l\ge-4m\pi+100$, we have $-4m\pi+100\le l<-8m\pi$. In this case, the following inequality holds
$$
\frac{e^{0.87m\pi}(l-4m\pi)^2}{l(l+4m\pi-8)}\le \frac{e^{0.87m\pi}(2l-100)^2}{92l}<\frac{e^{0.87m\pi}l}{23}<\frac{-8m\pi e^{0.87m\pi}}{23}<\frac{30}{4}.
$$
  Therefore by using \lemref{lk<l/2} we have $\tilde\alpha_{(l-2\lambda_{l,m}-8)/2,m}^{(1)}>\tilde\alpha_{l,m}^{(1)}$.

\end{proof}

\begin{lem}\label{lk>l/2}
Let $l\ge50$ be an even integer such that $r_l=2$. For an integer $m\le-1$ such that $l/2-\lambda_{l,m}\ge 50$, we have
$$
\alpha_{l,m}^{(1)}-\alpha_{(l-2\lambda_{l,m})/2,m}^{(1)}>\frac{0.02\pi}{(l-4m\pi)(l-4m\pi+8)}.
$$
\end{lem}
\begin{proof}
Using \eqref{k>l/2}, we have
\begin{align*}
b_{(l-2\lambda_{l,m})/2}(\alpha_{l,m}^{(1)})-b_{(l-2\lambda_{l,m})/2}(\alpha_{(l-2\lambda_{l,m})/2,m}^{(1)})&>\frac{\pi}{2}-
\frac{l+2\lambda_{l,m}}{4}\left(\frac{2\pi}{l-4m\pi\left(1-\frac{1}{2}\left(\frac{2\pi}{l-2\sqrt3m\pi}\right)^2\right)}\right)\\
&=\frac{\pi}{2}\left(1-\frac{l+2\lambda_{l,m}}{l-4m\pi\left(1-\frac{1}{2}\left(\frac{2\pi}{l-2\sqrt3m\pi}\right)^2\right)}\right).
\end{align*}
As $\lambda_{l,m}\le \lambda_{l,m}'\in J_{l,m}$, we have
\begin{equation}\label{0.005}
\lambda_{l,m}\le -2m\pi\left(1-\frac{1}{2}\left(\frac{2\pi}{l-2\sqrt3m\pi}\right)^2\right).
\end{equation}
For $l\ge 50$ and $m\le-1$, using \eqref{0.09} we have
$$
0<-m\pi\left(\frac{2\pi}{l-2\sqrt3m\pi}\right)^2<0.05
$$
Since $ \{-2m\pi\}\ge 0.06$, we have
$$
\lambda_{l,m}\le -2m\pi\left(1-\frac{1}{2}\left(\frac{2\pi}{l-2\sqrt3m\pi}\right)^2\right)-0.01
$$
Using this, we get that
\begin{align*}
b_{(l-2\lambda_{l,m})/2}(\alpha_{l,m}^{(1)})-b_{(l-2\lambda_{l,m})/2}(\alpha_{(l-2\lambda_{l,m})/2,m}^{(1)})&>\frac{\pi}{2}\left(1-\frac{l-4m\pi\left(1-\frac{1}{2}\left(\frac{2\pi}{l-2\sqrt3m\pi}\right)^2\right)-0.02}{l-4m\pi\left(1-\frac{1}{2}\left(\frac{2\pi}{l-2\sqrt3m\pi}\right)^2\right)}\right)\\
&=\frac{0.01\pi}{l-4m\pi\left(1-\frac{1}{2}\left(\frac{2\pi}{l-2\sqrt3m\pi}\right)^2\right)}.
\end{align*}
Using the mean value theorem, we get that for some $\xi\in (\pi/2,2\pi/3)$,
$$
\alpha_{l,m}^{(1)}-\alpha_{(l-2\lambda_{l,m})/2,m}^{(1)}>\frac{0.01\pi}{\left(l-4m\pi\left(1-\frac{1}{2}\left(\frac{2\pi}{l-2\sqrt3m\pi}\right)^2\right)\right)b_{(l-2\lambda_{l,m})/2}'(\xi)}
$$
Now 
$$
b_{(l-2\lambda_{l,m})/2}'(\xi)\le \frac{l-2\lambda_{l,m}-8m\pi}{4}.
$$
Using this, we get that
$$
\alpha_{l,m}^{(1)}-\alpha_{(l-2\lambda_{l,m})/2,m}^{(1)}>\frac{0.04\pi}{\left(l-4m\pi\left(1-\frac{1}{2}\left(\frac{2\pi}{l-2\sqrt3m\pi}\right)^2\right)\right)(l-2\lambda_{l,m}-8m\pi)}.
$$
As $\lambda_{l,m}>-2m\pi-4$, we get 
\begin{align*}
\alpha_{l,m}^{(1)}-\alpha_{(l-2\lambda_{l,m})/2,m}^{(1)}&>\frac{0.04\pi}{\left(l-4m\pi\left(1-\frac{1}{2}\left(\frac{2\pi}{l-2\sqrt3m\pi}\right)^2\right)\right)(l-4m\pi+8)}\\
&>\frac{0.04\pi}{(l-4m\pi)(l-4m\pi+8)}.
\end{align*}
\end{proof}

\begin{cor}\label{ck>l/2}
Let $l\ge50$ be an even integer such that $r_l=2$. For an integer $m\le-1$ such that $l/2-\lambda_{l,m}\ge 50$ we have
$\tilde\alpha_{l,m}^{(1)}>\tilde\alpha_{(l-2\lambda_{l,m})/2,m}^{(1)}$.
\end{cor}
\begin{proof} For any $l\ge 50$ and $m\le-1$, using \lemref{disp}, we have
\begin{equation}\label{4l/5}
\delta_{l,m,1}\le\frac{0.001\pi e^{0.87m\pi}}{l(l-2\sqrt3m\pi)}\le \frac{0.01\pi}{(l-4m\pi)(l-4m\pi+8)}.
\end{equation}
In view of \lemref{lk>l/2}, we need to show that
$$
 \delta_{(l-2\lambda_{l,m})/2,m,1}\le  \frac{0.03\pi}{(l-4m\pi)(l-4m\pi+8)}.
$$
{\bf Case 1.} Let $l/2-\lambda_{l,m}'< 50$. In this case $\lambda_{l,m}=l/2-52$.
     Therefore using \lemref{disp}, we have
 $$
    \delta_{(l-2\lambda_{l,m})/2,m,1}\le \frac{0.004\pi e^{0.87m\pi}}{(l-2\lambda_{l,m})(l-2\lambda_{l,m}-4\sqrt3 m\pi)}=\frac{0.004\pi e^{0.87m\pi}}{104(104-4\sqrt3 m\pi)}.
    $$
       Since $l/2-\lambda_{l,m}'<50$, we have $l<-4m\pi+100$, and so for $m\le-1$, we have 
    $$
  \frac{0.03\pi}{(l-4m\pi+8)(l-4m\pi)}\ge \frac{0.03\pi}{(-8m\pi+108)^2}>
  \frac{0.004\pi e^{0.87m\pi}}{104(104-4\sqrt3 m\pi)}.
    $$
{\bf Case 2.}
In this case we assume $l/2-\lambda_{l,m}'\ge50$. We have
 $ \lambda_{l,m}=\lambda_{l,m}'\in[-2m\pi-4,-2m\pi]$. Using this we get that
\begin{align*}
 (l-2\lambda_{l,m})(l-2\lambda_{l,m}-4\sqrt3 m\pi)&\ge (l+4m\pi)(l+4m\pi-4\sqrt3m\pi)\\&\ge l(l+4m\pi).
\end{align*}
So,
 $$
  \delta_{(l-2\lambda_{l,m})/2,m,1}\le \frac{0.004\pi e^{0.87m\pi}}{l(l+4m\pi)}.
 $$
And hence
 $$
 \frac{0.004\pi e^{0.87m\pi}}{l(l+4m\pi)}\le \frac{0.03\pi}{(l-4m\pi+8)^2} \ \text{ if and only if }\ \frac{e^{0.87m\pi}(l-4m\pi+8)^2}{l(l+4m\pi)}<\frac{30}{4}.
 $$
 First, we assume that $l>-5m\pi$. Then we have $-4m\pi<4l/5$, and then using this we get that
 $$
 \frac{e^{0.87m\pi}(l-4m\pi+8)^2}{l(l+4m\pi)}\le \frac{e^{0.87m\pi}(l+\frac{4l}{5}+8)^2}{l(l-\frac{4l}{5})}={20e^{0.87m\pi}}<\frac{30}{4}.
 $$
Next, we assume that $l<-5m\pi$. Since $l\ge-4m\pi+100$, in this case we have $-4m\pi+100\le l<-5m\pi$. And so
$$
\frac{e^{0.87m\pi}(l-4m\pi+8)^2}{l(l+4m\pi)}\le \frac{e^{0.87m\pi}(2l-92)^2}{100l}<\frac{e^{0.87m\pi}l}{25}<\frac{-5m\pi e^{0.87m\pi}}{25}<\frac{30}{4}.
$$
  Therefore by using \lemref{lk>l/2}, we obtain $\tilde\alpha_{(l-2\lambda_{l,m})/2,m}^{(1)}>\tilde\alpha_{l,m}^{(1)}$.

%
%
%
%
%
%

\end{proof}

\begin{lem}\label{2nd zero}
For even integers $l>k\ge50$, and $m\le -1$, we have 
$$
{\alpha}_{k,m}^{(2)}-{\alpha}_{l,m}^{(1)}\ge\frac{\pi}{k-4m\pi}.
$$
Moreover, $\tilde{\alpha}_{l,m}^{(1)}<\tilde{\alpha}_{k,m}^{(2)}$.
\end{lem}
\begin{proof}
Note that
$$
b_{k,m}({\alpha}_{k,m}^{(2)})\ge\frac{k\pi}{4}+\frac{3\pi}{2}
\ \text{ and } \
b_{l,m}({\alpha}_{l,m}^{(1)})\le\frac{l\pi}{4}+\pi.
$$
Also,
\begin{align*}
    b_{k,m}({\alpha}_{k,m}^{(2)})-b_{k,m}({\alpha}_{l,m}^{(1)})&=b_{k,m}({\alpha}_{k,m}^{(2)})-b_{l,m}({\alpha}_{l,m}^{(1)})+\frac{l-k}{2}{\alpha}_{l,m}^{(1)}\\
&\ge\frac{\pi}{2}+\frac{l-k}{2}\left({\alpha}_{l,m}^{(1)}-\frac{\pi}{2}\right)>\frac{\pi}{2}.
\end{align*}
By applying the mean value theorem and \eqref{bk'}, we get
$$
{\alpha}_{k,m}^{(2)}-{\alpha}_{l,m}^{(1)}\ge\frac{\pi}{k-4m\pi}.
$$
Hence by applying \lemref{disp} and \lemref{disp2}, for even integers $l,k\ge 50$, and $m\le-1$ we get
$
\tilde{\alpha}_{l,m}^{(1)}<\tilde{\alpha}_{k,m}^{(2)}$.
\end{proof}
\subsection{Concluding the proof} Corollaries \ref{c2,0}, \ref{cequal}, \ref{ck<l/2}, \ref{ck>l/2}, and \lemref{2nd zero}, along with \eqref{order}, give us the desired ordering mentioned at the end of \S\ref{struc}. This concludes the proof.

\section*{Acknowledgments}
The research of the author is supported by The Institute of Mathematical Sciences postdoctoral fellowship.
The author would like to thank The Institute of Mathematical Sciences for providing a friendly atmosphere for research.

\end{document}